\documentclass[dvipdf]{amsart}

\makeatletter
\@addtoreset{equation}{section}

\makeatother

\usepackage{amsfonts, amsmath, amssymb, amsthm}
\usepackage{url}
\usepackage[dvips]{graphicx}
\usepackage[dvips]{color}
\usepackage{amscd}
\usepackage[right]{lineno}
\usepackage{enumerate}

\newtheorem{theorem}{Theorem}[section]
\newtheorem{lemma}{Lemma}[section]

\newtheorem{corollary}{Corollary}[section]
\theoremstyle{remark}

\newtheorem{remark}{Remark}[section]

\newcommand{\teich}{\mathcal{T}}

\newcommand{\Bers}[1]{\mathcal{T}^B_{#1}}

\newcommand{\ve}{V_E}

\makeindex

\begin{document}

\title[Bergman kernel]{Estimates of the Bergman kernel on Teichm\"uller space}
\author{Guangming Hu and Hideki Miyachi}

\date{\today}
\address{
Guangming Hu,
College of Science,
Jinling Institute of Technology,
Nanjing, 211169, P.R. China.
}
\email{18810692738@163.com}
\thanks{The first author is partially supported by PhD research startup foundation of Jinling Institute Technology Grant Number  jit-b-202011 and jit-fhxm-2018}
\address{Hideki Miyachi,
School of Mathematics and Physics,
College of Science and Engineering,
Kanazawa University,
Kakuma-machi, Kanazawa,
Ishikawa, 920-1192, Japan
}
\email{miyachi@se.kanazawa-u.ac.jp}
\thanks{The second author is partially supported by JSPS KAKENHI Grant Numbers
20H01800,
16H03933,
17H02843}
\subjclass[2020]{Primary 32G05, 32G15, 32U35, 32A25, 32F45, Secondary 53C60}
\keywords{Teichm\"uller space, Teichm\"uller distance, Pluripotential theory, Bergman kernel}

\begin{abstract}
In this paper, a comparison between the Bergman kernel form and the pushforward measure of the Masur-Veech meaure on the Teichm\"uller space of genus $g\ge 2$ is obtained.
\end{abstract}

\maketitle

\section{Introduction}
\subsection{Background and Main theorem}
The Teichm\"uller space  $\teich_g$ of genus $g\ge 2$ is realized as a bounded domain in $\mathbb{C}^{3g-3}$ (cf. \cite{MR0130972}). In view of complex analysis, there is an important invariant $K_{\teich_g}$, the \emph{Bergman kernel form}, on $\teich_{g}$, which is defined from the reproducing kernel of the space of $L^2$-holomorphic $(3g-3)$-forms on $\teich_g$ (cf.  \cite{MR112162} and \cite{MR3887636}. See also \S\ref{subsec:Bergman_kernel}). From the transformation law, the Bergman kernel form is regarded as a non-negative $(6g-6)$-form on $\teich_g$ (cf. \eqref{eq:fundamental_eq2}).

The Teichm\"uller geodesic flow on the unit cotangent bundle $\mathcal{U}\teich_{g}$ on $\teich_{g}$ (in terms of the Teichm\"uller metric) admits a natural invariant measure, named the \emph{Masur-Veech measure} (cf. \cite{MR644018} and \cite{MR644019}). By descending the Masur-Veech measure to $\teich_g$ via the natural  projection $\mathcal{U}\teich_{g}\to \teich_g$, we get a measure $\mathbf{m}_g$ on $\teich_g$ which is invariant under the action of the mapping class group (cf. \cite{MR2913101}).

In this short paper, we shall show the following global comparison.

\begin{theorem}[Bergman kernel form and Masur-Veech measure]
\label{thm:main}
There are two positive constants $C_1$, $C_2$ depending only on $g$ such that
\begin{equation}
\label{eq:main}
C_1\,d\mathbf{m}_g\le
K_{\teich_{g}}
\le C_2\,d\mathbf{m}_g
\end{equation}
on $\teich_{g}$.
\end{theorem}
Notice that the inequality \eqref{eq:main} is understood as comparisons of non-negative $(6g-6)$-forms on $\teich_{g}$ (cf. \eqref{eq:fundamental_eq2}).
We will also prove that the Bergman kernel form on $\teich_g$ is comparable with the $(6g-6)$-dimensional Hausdorff measure with respect to the Teichm\"uller distance (cf. Corollary \ref{coro:hausdorff_measure}). Our comparisons give global geometric informations of the Bergman kernel form in view of the Teichm\"uller theory.

As a prior work, B. Chen \cite{MR2106264} studied the asymptotic behavior of the Bergman kernel function at the Bers boundary. However, to the authors' knowledge, there are less informations on the Bergman kernel form on the Teichm\"uller space. On the other hand, there is an enormous amount of studies of the Bergman kernel in the function theory of several complex variables, and many applications to several fields in mathematics. The authors believe that the study of the Bergman kernel on the Teichm\"uller space is important for developing the complex analytical aspect in the Teichm\"uller theory.

\subsection{Two corollaries}
For $x\in \teich_g$, let 
$B_T(x,R)$
be the open $R$-ball of the center at $x$ with respect to the Teichm\"uller distance.
Applying \cite[Theorem 1.3]{MR2913101}, we obtain
\begin{corollary}[Volume estimate]
\label{coro:volume-estimate}
There are two positive constants $D_1$, $D_2$ depending only on $g\ge 2$ such that
\begin{align*}
D_1&\le \liminf_{R\to \infty}e^{-(3g-3)R}\int_{B_T(x,R)}K_{\teich_g}
\le
\limsup_{R\to \infty}e^{-(3g-3)R}\int_{B_T(x,R)}K_{\teich_g}\le D_2
\end{align*}
for any $x\in \teich_g$.
\end{corollary}
It is an interesting problem to determine whether the limit of $\displaystyle e^{-(3g-3)R}\int_{B_T(x,R)}K_{\teich_g}$ as $R\to \infty$ exists or not, like Athreya,  Bufetov, Eskin, and Mirzakhani's work on $\mathbf{m}_g$ in \cite{MR2913101}.

Since the Bergman kernel form is a biholomorphic invariant, 
$K_{\teich_{g}}$ descends to a non-negative $(6g-6)$-form on the moduli space $\mathcal{M}_g$ of genus $g\ge 2$.
Since the volume of $\mathcal{M}_g$ with respect to $\mathbf{m}_g$ is finite (cf. \cite{MR644018} and \cite{MR2913101}), we obtain

\begin{corollary}[Integral is finite]
\label{coro:finite_volume}
For any $g\ge 2$,
$$
\int_{\mathcal{M}_g}K_{\teich_g}<\infty.
$$
\end{corollary}

\subsection{About the proof of Theorem \ref{thm:main}}
As discussed in \S\ref{subsec:Masur-Veech_measure},
Theorem \ref{thm:main} follows from a comparison between the Bergman kernel form and the Busemann volume form as given in Theorem \ref{thm:Bergman_indicatrix} and Dowdall-Duchin-Masur's comparisons of several measures on the Teichm\"uller space in \cite{MR3213829}. As remarked after the statement of Theorem \ref{thm:Bergman_indicatrix}, the idea of the proof of Theorem \ref{thm:Bergman_indicatrix} is given by mimiking that of the proof by Z. B{\l}ocki \cite{MR3364678} for convex domains.

Theorem \ref{thm:Bergman_indicatrix} is stated for the Teichm\"uller space $\teich_{g,m}$ of type $(g,m)$ with $2g-2+m>0$. The authors hope that Dowdall-Duchin-Masur's comparisons of volumes also hold for all $(g,m)$ with $2g-2+m>0$, and that the estimate \eqref{eq:main}  for any $(g,m)$ does.


\section{Teichm\"uller theory}
\subsection{Teichm\"uller space}
The \emph{Teichm\"uller space} $\teich_{g,m}$ of type $(g,m)$
is the equvalence classes
of marked Riemann surfaces of type $(g,m)$. 
A \emph{marked Riemann surface} $(M,f)$ of type $(g,m)$
is a pair of a Riemann surface $M$ of analytically finite type $(g,m)$
and an orientation preserving homeomorphism
$f\colon \Sigma_{g,m}\to M$.
Two marked Riemann surfaces $(M_{1},f_{1})$ and $(M_{2},f_{2})$
of type $(g,m)$ are
\emph{(Teichm\"uller) equivalent} if there is a conformal mapping
$h\colon M_{1}\to M_{2}$ such that $h\circ f_{1}$ is homotopic to $f_{2}$. For simplicity, we write $(M,f)$ the Teichm\"uller equivalence class of the marked Riemann surface $(M,f)$.

Let $x=(M,f)\in \teich_{g,m}$.
Let $L^\infty(M)$ be the space of measurable $(-1,1)$-forms $\mu=\mu(z)d\overline{z}/dz$ on $M$ with
$$
\|\mu\|_\infty={\rm ess.sup}\{|\mu(p)|\mid p\in M\}<\infty.
$$
Let $B_x=\{\mu\in L^\infty(M)\mid \|\mu\|_\infty<1\}$. For $\mu\in B_{x}$, let $f_\mu=f_{\mu;x}$ be the $\mu$-quasiconformal mapping on $M$. We define the \emph{Bers projection} $\Phi_{x}\colon B_{x}\to \teich_{g,m}$ by $\Phi_{x}(\mu)=(f_\mu(M),f_\mu\circ f^{-1})\in \teich_{g,m}$.

For $x=(M,f)\in \teich_{g,m}$,
we denote by $\mathcal{Q}_{x}$ the complex Banach space
of holomorphic quadratic differentials $q=q(z)dz^{2}$
on $M$ with
$$
\|q\|=\int_{M}|q(z)|\,\dfrac{i}{2}dz\wedge d\overline{z}
<\infty,
$$
The holomorphic tangent space $T_x\teich_{g,m}$ at $x\in \teich_{g,m}$ is identified with the quotient space of $L^\infty(M)$ by the equivalence relation, where $\mu_1$, $\mu_2\in L^\infty(M)$ are equivalent if 
$$
\displaystyle\int_M\mu_1 q=\int_M\mu_2 q
$$
for all $q\in\mathcal{Q}_x$. In other words, the kernel of the differenital $(\Phi_{x})_*$ at $x$ consists of $\mu\in L^\infty(M)$ with $\displaystyle\int_{M}\mu q=0$ for all $q\in \mathcal{Q}_x$ (cf. \cite[Theorem 7.6]{MR1215481}). We denote by $[\mu]\in T_{x}\teich_{g,m}$ the tangent vector at $x$ associated to $\mu\in L^\infty(M)$. By definition, $(\Phi_x)_*(\mu)$ represents the equivalence class of $\mu$. Namely,
\begin{equation}
\label{eq:Bers_projection_derivative}
(\Phi_x)_*(\mu)=[\mu]\in T_x\teich_{g,m}
\end{equation}
for $\mu\in L^\infty(M)=T_0 B_x$.

\subsection{Teichm\"uller distance}
\label{subsec:teich_distance}
The \emph{Teichm\"uller distance} $d_T$ is a complete distance
on $\teich_{g,m}$ defined by
$$
d_T(x_1,x_2)=\frac{1}{2}\log \inf_hK(h)
$$
for $x_i=(M_i,f_i)$ ($i=1,2$),
where the infimum runs over all quasiconformal mappings
$h\colon M_1\to M_2$ homotopic to $f_2\circ f_1^{-1}$
and $K(h)$ is the maximal dilatation of $h$.
%
%
The Teichm\"uller distance is known to be a Finsler distance with the Finsler metric
$$
F_x(v)=\sup_{\|q\|=1,q\in \mathcal{Q}_x}\left|\int_M\mu q\right|
$$
for $x=(M,f)\in \teich_{g,m}$, $v=[\mu]\in T_x\teich_{g,m}$ and $\mu\in L^\infty(M)$ (cf. \cite[\S7]{MR903027}). We call $F_x$ the \emph{Teichm\"uller metric}. It is known that the Teichm\"uller metric coincides with the Kobayashi-Royden metric (cf. \cite{MR0288254} and \cite{MR0430319}).

The \emph{Kobayashi-Teichm\"uller indicatrix} $\mathcal{I}_T(x)$ at $x\in \teich_{g,m}$ is defined by
$$
\mathcal{I}_T(x)=\{v\in T_x\teich_{g,m}\mid F_x(v)\le 1\}.
$$
Since the Teichm\"uller metric $F_x$ is a norm on $T_x\teich_{g,m}$, the indicatrix $\mathcal{I}_T(x)$ is a convex set in $T_x\teich_{g,m}$.

\subsection{Bers slice}
Let $x_0=(M_0,f_0)\in \teich_{g,m}$. Let $\Gamma_0$ be the Fuchsian group of $M_0$ acting on the upper-half plane $\mathbb{H}^2$. Let $A_2=A_2(\mathbb{H}^*,\Gamma_0)$ be the set of holomorphic functions $\varphi$ on the lower-half plane $\mathbb{H}^*$ such that $\varphi(\gamma(z))\gamma'(z)^2=\varphi(z)$ $(z\in \mathbb{H}^*$, $\gamma\in \Gamma_0)$ and
$$
\|\varphi\|_\infty=\sup_{z\in \mathbb{H}^*}4\,{\rm Im}(z)^2|\varphi(z)|<\infty.
$$

For $\varphi\in A_2$, we define a locally univalent function $W_\varphi$ on $\mathbb{H}^*$ such that $W_\varphi(z)=(z+i)^{-1}+o(1)$ as $z\to -i$ and the Schwarzian derivative of $W_\varphi$ coincides with $\varphi$. The Teichm\"uller space $\teich_{g,m}$ is canonically identified with
the \emph{Bers slice} $\Bers{x_0}$ which consists of $\varphi\in A_2$ such that $W_\varphi$ admits a quasiconformal extension on the Riemann sphere, via Bers' simultaneous uniformization (cf. \cite{MR0111834} and \cite{MR257351}). The biholomorphic identification $\beta_{x_0}\colon \teich_{g,m}\to\Bers{x_0}\subset A_2$ is called the \emph{Bers embedding}.

\subsection{Ahlfors-Weill section}
Let $x_0=(M_0,f_0)\in \teich_{g,m}$.
We define $H_{x_0}\colon A_2\to L^\infty(M_0)$ by
\begin{equation}
\label{eq:Ahlfors_Weill_section}
H_{x_0}(\psi)(z)=-2\,{\rm Im}(z)^2\psi(\overline{z}).
\end{equation}
Then $\|H_{x_0}(\psi)\|_\infty=\|\psi\|_\infty/2$. 

Let $B_{r;x_0}^\infty$ be the closed $r$-ball in $A_2$ with respect to $\|\cdot \|_\infty$.
The restriction of $H_{x_0}$ to the ball ${\rm Int}(B^\infty_{2;x_0})$ is called the \emph{Ahlfors-Weill section} which satisfies
\begin{align}
\beta_{x_0}\circ \Phi_{x_0}\circ H_{x_0}(\varphi)
&=\varphi\quad (\varphi\in {\rm Int}(B^\infty_{2;x_0}))\label{eq:Bers-projection1} \\
(\beta_{x_0})_*([H_{x_0}(\psi)])&=\psi\quad (\psi\in T_0A_2=A_2)
\label{eq:Bers-projection2}
\end{align}
 (cf. \cite{MR148896} and \cite[Theorem 6.9]{MR1215481}).
Via the linear isomorphism $H_{x_0}$, the Kobayashi-Teichm\"uller indicatrix is realized in $A_2=T_0A_2$ as
$$
H_{x_0}^{-1}(\mathcal{I}_T(x))=\{\psi\in A_2\mid F_x([H_{x_0}(\psi)])\le 1\}.
$$
%

%
We claim

\begin{lemma}
\label{lem:ball_A_2_indicatrix}
$B^\infty_{2;x_0}\subset H_{x_0}^{-1}(\mathcal{I}_T(x_0)))\subset B^\infty_{6;x_0}$.
\end{lemma}

\begin{proof}
Since $\psi\in \partial  H_{x_0}^{-1}(\mathcal{I}_T(x_0))$ satisfies $F_{x_0}([H_{x_0}(\psi)])=1$,
$$
1=F_{x_0}([H_{x_0}(\psi)])=\sup_{\|q\|=1} \left|\int_{M_0}H_{x_0}(\psi)q\right|
\le \|H_{x_0}(\psi)\|_\infty=\|\psi\|_\infty/2
$$
for all $\psi\in \partial  H_{x_0}^{-1}(\mathcal{I}_T(x_0))$.
This means that $B^\infty_{2;x_0}\subset H_{x_0}^{-1}(\mathcal{I}_T(x_0))$.

Let $\psi\in H_{x_0}^{-1}(\mathcal{I}_T(x_0))\subset A_2=T_0A_2$. 
By Nehari-Kraus' theorem, the image of the Bers embedding is contained in the ball $B^\infty_{6;x_0}$ (cf. \cite{MR1215481}). Since the Teichm\"uller metric coincides with the Kobayashi metric, by the distance decreasing property and \eqref{eq:Bers-projection2}, we have
$$
\dfrac{\|\psi\|_\infty}{6}=\dfrac{\|(\beta_{x_0})_*([H_{x_0}(\psi)])\|_\infty}{6}
\le F_{x_0}([H_{x_0}(\psi)])\le 1,
$$
where the left-hand side of the above calculation is the Kobayashi-Finsler norm of $\psi\in T_0B^\infty_{6;x_0}=T_0A_2$ on $B^\infty_{6;x_0}$ (e.g. \cite{MR1033739}). This means that $\psi\in B^\infty_{6;x_0}$.
\end{proof}

\subsection{Differential of maximal dilatation}
For $x,y\in \teich_{g,m}$, we set
\begin{equation}
\label{eq:green_2}
k_0(x,y)=\tanh d_T(x,y).
\end{equation}
The following lemma immediately follows from the discussion in the proof of \cite[\S6.6, Theorem 7]{MR903027}. For the completeness, we shall give a proof.
\begin{lemma}
\label{lem:Gardiner_book}
Let $x_0\in \teich_{g,m}$.
For $v=[\mu]\in T_{x_0}\teich_{g,m}$,
let $x_t$ be the quasiconformal deformation of $x_0$ associated to the Beltrami differential $t\mu$.
Then, 
\begin{equation}
\label{eq:T-distance1}
\left|k_0(x_0,x_t)-tF_{x_0}(v)\right|\le 4t^2\|\mu\|_\infty^2
\end{equation}
when $0\le t<(2\|\mu\|_\infty)^{-1}$.
\end{lemma}

\begin{proof}
We may assume that $\mu\ne 0$. By definition, $k_0(t):=k_0(x_0,x_t)\le t\|\mu\|_\infty$ when $0\le t<1/\|\mu\|_\infty$.
Following \cite[\S6.4]{MR903027}, we set
\begin{align*}
I[\mu]&=\sup_{\|q\|=1}\left|\int_{M_0}\dfrac{\mu q}{1-|\mu|^2}\right|, \\
J[\mu]&=\sup_{\|q\|=1}\int_{M_0}\dfrac{|\mu|^2|q|}{1-|\mu|^2}.
\end{align*}
By a simple calculation, we have
$$
|I[t\mu]-tF_{x_0}(v)|\le \dfrac{t^3\|\mu\|_\infty^3}{1-t^2\|\mu\|_\infty^2}.
$$
Combining with \cite[\S6.4, Theorem 4]{MR903027}, we obtain
\begin{align}
\left|
k_0(x_0,x_t)-tF_{x_0}(v)
\right|
&\le
\left|
k_0(t)-
\dfrac{k_0(t)}{1-k_0(t)^2}
\right|
+
\left|
\dfrac{k_0(t)}{1-k_0(t)^2}-I[t\mu]
\right|
+
\left|
I[t\mu]-tF_{x_0}(v)
\right| 
\nonumber
\\
&\le
\dfrac{k_0(t)^3}{1-k_0(t)^2}+
J[t\mu]+\dfrac{k_0(t)^2}{1-k_0(t)^2}
+
\dfrac{t^3\|\mu\|_\infty^3}{1-t^2\|\mu\|_\infty^2} 
\nonumber
\\
&\le
\dfrac{2t^2\|\mu\|_\infty^2+2t^3\|\mu\|_\infty^3}{1-t^2\|\mu\|_\infty^2}
\le 4t^2\|\mu\|_\infty^2
\label{eq:teichmuller-distance}
\end{align}
when $0\le t\le (2\|\mu\|_\infty)^{-1}$.
\end{proof}

\section{Complex analysis}

\subsection{Bergman Kernel}
\label{subsec:Bergman_kernel}
Let $\Omega$ be an $N$-dimensional complex manifold. We denote by $\mathcal{O}L^2_{N,0}(\Omega)$ the Hilbert space of holomorphic $N$-forms $f=f(z)dZ$ ($dZ=dz_1\wedge\cdots\wedge dz_N$) with the inner product
\begin{equation}
\label{eq:inner_product}
(f_1,f_2)=\dfrac{i^{N^2}}{2^N}\int_\Omega f_1\wedge \overline{f_2}
=\int_{\Omega}f_1(z)\overline{f_2(z)}\,dV_E(z),
\end{equation}
where $dV_E=dx_1dy_1\cdots dx_Ndy_N$ is the standard Euclidean measure (Lebesgue measure) on local charts $(z_1,\ldots,z_N)$
and  $z_k=x_k+iy_k$ $(1\le k\le N)$.
The \emph{reproducing kernel form} on $\mathcal{O}L^2_{N,0}(\Omega)$ is a bi-form on $\Omega$ defined by
$$
K_\Omega(z,w)=\sum_{k=1}^\infty f_k\otimes \overline{f_k}=\sum_{k=1}^\infty f_k(z)\overline{f_k(w)}\,
dZ\otimes d\overline{W},
$$
where $d\overline{W}=d\overline{w}_1\wedge \cdots \wedge d\overline{w}_N$ and $\{f_k=f_k(z)dZ\}_{k=1}^\infty$ is a complete orthonormal basis of $\mathcal{O}L^2_{N,0}(\Omega)$
(cf. \cite{MR112162} and \cite[Chapter 4]{MR3887636}). We call $K_\Omega=K_\Omega(z,z)$ the \emph{Bergman kernel form} on $\Omega$:
\begin{align}
K_\Omega
&=K_\Omega(z)\,
dZ\otimes d\overline{Z}
=
\sum_{k=1}^\infty |f_k(z)|^2\,
dZ\otimes d\overline{Z}
\label{eq:Bergman_kernel_form1}
\end{align}
on $\Omega$. The transformation law (cf. \cite[(4.9)]{MR3887636})
$$
K_{\Omega'}(F(z))|\det F'(z)|^2=K_\Omega(z)\quad (z\in \Omega)
$$
with a biholomorphic mapping (or a local chart) $F\colon \Omega\to \Omega'$
($F'$ is the complex Jacobian of $F$)
implies that
\begin{equation}
\label{eq:fundamental_eq2}
K_\Omega=K_\Omega(z)dV_E
\end{equation}
is a well-defined non-negative $2N$-form on $\Omega$. It is known that 
\begin{equation}
\label{eq:fundamental_eq3}
K_\Omega\le 
K_{\Omega'}
\end{equation}
for any open set $\Omega'\subset \Omega$ (cf. \cite[Corollary 4.1]{MR3887636}).

When $\Omega$ is a domain in $\mathbb{C}^N$, the space $\mathcal{O}L^2_{N,0}(\Omega)$ is isometrically identified with the space $\mathcal{O}L^2_{0,0}(\Omega)$ of $L^2$-holomorphic functions with respect to the standard Euclidean measure (Lebesgue measure) by
$$
\mathcal{O}L^2_{N,0}(\Omega)\ni f=f(z)dZ\mapsto f(\cdot )\in \mathcal{O}L^2_{0,0}(\Omega),
$$
(cf. \cite[Example 4.1]{MR3887636}). The \emph{Bergman kernel function} on a domain $\Omega$ in $\mathbb{C}^N$ is defined by
\begin{equation}
\label{eq:Bergman_kernel_function}
\sum_{k=1}^\infty|f_k(z)|^2
\end{equation}
where $\{f_k\}_{k=1}^\infty$ is a complete orthonormal basis of $\mathcal{O}L^2_{0,0}(\Omega)$ (cf. \cite[\S4.1.1]{MR3887636}). Hence, when $\Omega$ is a domain, from \eqref{eq:inner_product} and \eqref{eq:Bergman_kernel_form1}, the coefficient $K_\Omega(z)$ of the Bergman kernel form on $\Omega$ coincides with the Bergman kernel function \eqref{eq:Bergman_kernel_function} on $\Omega$.

%

\subsection{Pluricomplex Green function}
Let $\Omega$ be a domain in $\mathbb{C}^N$.
The \emph{pluricomplex Green function} $g_\Omega$ with a pole at $w\in \Omega$ is defined by
$$
g_\Omega(w,z)=\sup\{u(z)\in {\rm PSH}(\Omega)^-\mid \limsup_{z\to w}(u(z)-\log|z-w|)<\infty\}
$$
where ${\rm PSH}(\Omega)^-$ denotes the class of negative plurisubharmonic functions on $\Omega$ (cf. \cite{MR820321}). In \cite{MR1142683} and \cite{MR4028456}, it is shown that
\begin{equation}
\label{eq:green_function}
g_{\teich_{g,m}}(x,y)=\log \tanh d_T(x,y)=\log k_0(x,y)
\end{equation}
for $x,y\in \teich_{g,m}$.

\section{Estimates of the Bergman Kernel}
\label{sec:estimate_Bergman_kernel}
\subsection{Busemann volume forms}
We first recall the \emph{Busemann volume form} on an $N$-dimensional Finsler manifold $(M,F)$ after \cite{MR2132656} and \cite[\S4]{MR3213829}. Usually, the Finsler norm is assumed to be smooth. However, we only assume here the Finsler norm to be continuous.

Let $x\in M$ and $B_x=\{v\in T_xM\mid F_x(v)\le 1\}$ be the unit ball (the $F$-indicatrix) with respect to the Finsler norm $F$. For an identification $T_xM\cong \mathbb{R}^N$ induced by a local coodinate around $x$, we define the \emph{Busemann volume form} on $M$ by
\begin{equation}
\label{eq:Buseman-volume-form}
d\mu_{M;B}=\dfrac{\epsilon_N}{V_E(B_x)}dV_E,
\end{equation}
where $\epsilon_N$ is the volume of the unit ball in $\mathbb{R}^N$ and $V_E$ is the standard Euclidean measure (Lebesgue measure) on $\mathbb{R}^N$ as the previous section (``B" in the subscription of the notation stands for the initial letter of ``Busemann").

\subsection{Comparison}
In this section, we show 
\begin{theorem}[Bergman kernel form and Busemann volume form]
\label{thm:Bergman_indicatrix}
\begin{equation}
\label{eq:Bergman_indicatrix}
\dfrac{1}{\epsilon_{6g-6+2m}}d\mu_{\teich_{g,m};B}\le
K_{\teich_{g,m}}
\le
\dfrac{3^{6g-6+2m}}{\epsilon_{6g-6+2m}}d\mu_{\teich_{g,m};B}
\end{equation}
on $\teich_{g,m}$.
\end{theorem}
As discussed in \eqref{eq:fundamental_eq2},
the inequality \eqref{eq:Bergman_indicatrix} is regarded as comparisons of non-negative $(6g-6+2m)$-forms on $\teich_{g,m}$.

In the pluripotential theory, the relation \eqref{eq:Bergman_indicatrix} is first observed by Z. B{\l}ocki  for convex domains in $\mathbb{C}^n$ (cf. \cite[Theorem 2]{MR3364678}). The inequality \eqref{eq:Bergman_indicatrix} is closely related to the inequality conjectured by Suita \cite{MR367181} (cf. \cite{MR3318425}).

We give two proofs of the lower bound of Theorem \ref{thm:Bergman_indicatrix}.
The first proof is based on the same line as his proof in \cite{MR3364678}, meanwhile we apply Teichm\"uller theory in the essential part of the proof.  The second proof is given by characterizing the Azukawa metric on the Teichm\"uller space and applying B{\l}ocki-Zwonek's result in \cite{MR3318425}.
The discussion of the upper estimate is a mimic of the discussion by B{\l}ocki in \cite[Theorem 5]{MR3364678} for the convex domains (see also \cite[Corollary 4]{MR4011523}).

\subsection{First proof of lower estimate}
\label{subsec:1st_proof_lower}
It suffices to confirm the equation \eqref{eq:Bergman_indicatrix}  for a local chart at $x_0$. 

We fix a complex linear identification $L\colon A_2\cong \mathbb{C}^{3g-3+m}$. Then $L\circ \beta_{x_0}\colon \teich_{g,m}\to \mathbb{C}^{3g-3+m}$ is a complex local chart at $x_0$ with $L\circ \beta_{x_0}(x_0)=0\in \mathbb{C}^{3g-3+m}$. We denote by $\ve$ the Euclidean volume form (Lebesgue measure) on $\mathbb{C}^{3g-3+m}$ as above. From \eqref{eq:Bers-projection2}, the coordinate $L\circ \beta_{x_0}$ induces a complex linear isomorphism
$$
L\circ ((\Phi_{x_0})_*\circ H_{x_0})^{-1}\colon T_{x_0}\teich_{g,m}\to \mathbb{C}^{3g-3+m},
$$
which induces the Euclidean volume form on $T_{x_0}\teich_{g,m}$. For simplicity, we denote by $V_E$ the volume form on $T_{x_0}\teich_{g,m}$.
%

As remarked in the previous section, the following lemma is first observed by B{\l}ocki \cite[Proposition 3]{MR3364678} for convex domains in the complex Euclidean space by applying Lempert's theory \cite{MR660145}.

\begin{lemma}[Volume of sublevel sets of pluricomplex Green function]
\label{lem:volume_indicatrix}
Under the above identifications $T_{x_0}\teich_{g,m}\cong\mathbb{C}^{3g-3+m}$ and $\teich_{g,m}\cong \Bers{x_0}\subset A_2\cong \mathbb{C}^{3g-3+m}$, we have
$$
\lim_{a\to \infty}e^{-2(3g-3+m)a}V_E(\{y\in \teich_{g,m}\mid g_{\teich_{g,m}}(x_0,y)<-a\})=\ve(\mathcal{I}_T(x_0)).
$$
\end{lemma}

\begin{proof}
Take $a>0$ such that $\beta_{x_0}(\{y\in \teich_{g,m}\mid g_{\teich_{g,m}}(x_0,y)\le-a\})\subset {\rm Int}(B^\infty_{1/3;x_0})$.
Let $\mathcal{S}_{x_0}=\partial (H_{x_0}^{-1}(\mathcal{I}_T(x_0)))\subset A_2\cong \mathbb{C}^{3g-3+m}$, 
and $\psi\in \mathcal{S}_{x_0}$. 
Since $\|\psi\|_\infty\ge 2$ by Lemma \ref{lem:ball_A_2_indicatrix},
from \eqref{eq:Bers-projection1},
any solution of the equation $k_0(x_0,\Phi_{x_0}(tH_{x_0}(\psi)))=e^{-a}$ on $t$ satisfies $0\le t<1/6$.
Since $F_{x}([H_{x_0}(\psi)])=1$ and $\|H_{x_0}(\psi)\|_\infty=\|\psi\|_\infty/2\le 3$, from Lemma \ref{lem:Gardiner_book}, we have
$$
|e^{-a}-t|\le 144t^2.
$$
From \eqref{eq:green_function}, we obtain
\begin{align*}
\{t\psi\mid 0\le t< \delta_2(a),\psi\in \mathcal{S}_{x_0}\}
&\subset \beta_{x_0}(\{y\in \teich_{g,m}\mid g_{\teich_{g,m}}(x_0,y)<-a\}) \\
&\subset\{t\psi\mid 0\le t< \delta_1(a),\psi\in \mathcal{S}_{x_0}\},
\end{align*}
when $a>0$ is sufficiently large,
where $\delta_1(a)=(1-\sqrt{1-576e^{-a}})/288$ and $\delta_2(a)=(\sqrt{1+576e^{-a}}-1)/288$. Since $\delta_i(a)=e^{-a}+O(e^{-2a})$ as $a\to \infty$ for $i=1,2$, we have
$$
V_E(\{g_{\teich_{g,m}}(x_0,y)<-a\})
=e^{-2(3g-3+m)a}\ve(\mathcal{I}_T(x_0))+o(e^{-2(3g-3+m)a})
$$
as $a\to \infty$.
\end{proof}

Let us finish the proof of the lower estimate in Theorem \ref{thm:Bergman_indicatrix}.
 B{\l}ocki \cite[Theorem 1]{MR3364678} observed that
$$
K_\Omega(z)\ge \dfrac{1}{e^{2Na}V_E(\{g_{\Omega}(z,\cdot)<-a\})}
$$
for any pseudoconvex domain $\Omega\subset \mathbb{C}^N$ and $z\in \Omega$. Notice that B{\l}ocki \cite{MR3364678} considered the Bergman kernel as a reproducing kernel function on the space of $L^2$-holomorphic functions with respect to the Lebesgue measure $dV_E$ (cf. \eqref{eq:Bergman_kernel_function}).
Since our inner product is defined as \eqref{eq:inner_product}, by \eqref{eq:fundamental_eq2}, the Bergman kernel function on $\Bers{x_0}\subset \mathbb{C}^{3g-3+m}$ appears as the coefficient of the $(6g-6+2m)$-form $K_{\teich_{g,m}}$ in terms of the chart discussed in the beginning of this section.
Hence, we get the desired inequality \eqref{eq:Bergman_indicatrix} from Lemma \ref{lem:volume_indicatrix} by letting $a\to \infty$.

\subsection{Second proof of lower estimate}
Let $\Omega$ be an $N$-dimensional complex manifold. The \emph{Azukawa metric} on $\Omega$ is defined by
$$
A_\Omega(p;v)=\limsup_{t\to 0}\dfrac{\exp(g_{\Omega}(p,\varphi(t)))}{|t|}
$$
for $v\in T_p\Omega$ and $\varphi\colon \{|t|<\epsilon\}\to \Omega$ is a holomorphic map with $\varphi(0)=p$ and $\varphi_*(\partial/\partial t\mid_{t=0})=v$ (cf. \cite[\S2]{MR879385}).

\begin{lemma}
\label{lem:Azukawa=Teichmuller}
The Azukawa metric on $\teich_{g,m}$ coincides with the Teichm\"uller metric.
\end{lemma}

\begin{proof}
Let $x_0\in \teich_{g,m}$ and $v\in T_{x_0}\teich_{g,m}$. Let $\varphi\colon\{|t|<\epsilon\}\to \teich_{g,m}$ with $\varphi(0)=p$ and $\varphi_*(\partial/\partial t\mid_{t=0})=v$.
By Lemma \ref{lem:Gardiner_book} and \eqref{eq:green_function},
$$
\exp(g_{\teich_{g,m}}(x_0,\varphi(t)))=
k_0(x_0,\varphi(t))=|t|F_{x_0}(v)+o(|t|),
$$
and $A_{\teich_{g,m}}(x_0;v)=F_{x_0}(v)$.
\end{proof}

Let us finish the second proof of the lower estimate in Theorem \ref{thm:Bergman_indicatrix}. In \cite{MR3318425}, B{\l}ocki and Zwonek proved that the Bergram kernel function is at least the reciprocal of the volume of the Azukawa indicatrix for pseudoconvex domains. This implies the desired estimate.

\begin{remark}
\label{rem:remark1}
The second proof can be applied for general situations. Indeed, the inequality \eqref{eq:Bergman_indicatrix} holds for pseudoconvex domains when the Busemann volume form is defined with the Azukawa metric instead of the Kobayashi (Teichm\"uller) metric.
\end{remark}

\subsection{An upper estimate of the Bergman kernel}
We fix a local chart $L\circ\beta_{x_0}$ on $\teich_{g,m}$ as \S\ref{subsec:1st_proof_lower}.
From Lemma \ref{lem:ball_A_2_indicatrix} and Nehari's theorem (cf. \cite{MR1215481}), 
\begin{equation}
\label{eq:comp_Indicatrix_Tg}
H_{x_0}^{-1}(\mathcal{I}_T(x_0))\subset 3\Bers{x_0},
\end{equation}
where $rE=\{r\psi\in A_2\mid \psi\in E\}$ for $E\subset A_2$.
Notice that $H_{x_0}^{-1}(\mathcal{I}_T(x_0))$ is balanced in the sense that $\lambda \psi\in H_{x_0}^{-1}(\mathcal{I}_T(x_0))$ for $\psi\in H_{x_0}^{-1}(\mathcal{I}_T(x_0))$ and $|\lambda|\le 1$. Furthermore, $H_{x_0}^{-1}(\mathcal{I}_T(x_0))$ is convex, and hence is pseudoconvex. Therefore
$$
K_{(1/3)H_{x_0}^{-1}(\mathcal{I}_T(x_0))}
=\dfrac{dV_E}{V_E((1/3)H_{x_0}^{-1}(\mathcal{I}_T(x_0)))}
$$
at the origin of $A_2$,
after recognizing the Bergman kernel form as a non-negative $(6g-6+2m)$-form
 (e.g. \cite{MR4011523}). From \eqref{eq:fundamental_eq3}
 and \eqref{eq:comp_Indicatrix_Tg}, we conclude
\begin{align*}
K_{\teich_{g,m}}
&\le 
K_{(1/3)H_{x_0}^{-1}(\mathcal{I}_T(x_0))}
=\dfrac{dV_E}{V_E((1/3)H_{x_0}^{-1}(\mathcal{I}_T(x_0)))}
\\
&=\dfrac{3^{6g-6+2m}dV_E}{V_E(H_{x_0}^{-1}(\mathcal{I}_T(x_0)))}
=\dfrac{3^{6g-6+2m}}{\epsilon_{6g-6+2m}}d\mu_{\teich_{g,m};B}
\end{align*}
at $x_0\in \teich_g$.

\subsection{Proof of Theorem \ref{thm:main}}
\label{subsec:Masur-Veech_measure}
In \cite[Corollary 4.4]{MR3213829}, Dowdall, Duchin, and Masur observed that the pushforward measure $\mathbf{m}_g$ of the Masur-Veech measure via the projection $\mathcal{U}\teich_{g}\to \teich_g$ is comparable with the Busemann volume form. Therefore, we obtain the estimate in Theorem \ref{thm:main}.

\subsection{Hausdorff measure on $\teich_{g,m}$}
In \cite{MR3213829}, Dowdall, Duchin, and Masur also noticed that the Busemann volume form on $\teich_{g}$  associated to the Teichm\"uller metric coincides with the $(6g-6)$-dimensional Hausdorff measure $\mathcal{H}_{\teich_{g}}$ associated to the Teichm\"uller metric. This coincidence also holds for $\teich_{g,m}$ since Busemann proved that the Busemann volume coincides with the top-dimensional Hausdorff measure with respect to the Finsler distance for arbitrary (continuous) Finsler manifolds (cf. \cite{MR20626} and \cite[Theorem 3.23]{MR2132656}).

\begin{corollary}[Hausdorff measure]
\label{coro:hausdorff_measure}
$$
 \dfrac{1}{\epsilon_{6g-6+2m}}d\mathcal{H}_{\teich_{g,m}}
 \le 
K_{\teich_{g,m}}\le \dfrac{3^{6g-6+2m}}{\epsilon_{6g-6+2m}}d\mathcal{H}_{\teich_{g,m}}
$$
on $\teich_{g,m}$.
\end{corollary}


\begin{remark}
The coincidence $\mu_{\Omega;B}=\mathcal{H}_\Omega$ for the Kobayashi distance is 
also directly observed by Bland and Graham \cite[Theorem 1]{MR848839} for Kobayashi hyperbolic manifolds $\Omega$ with continuous infinitesimal Kobayashi-Royden metrics whose Kobayashi-indicatrices are convex. 
(In our case, the Teichm\"uller metric $F_{x_0}$ is a norm on the tangent space.)
\end{remark}

\begin{remark}
As remarked in Remark \ref{rem:remark1}, the inequality in Corollary \ref{coro:hausdorff_measure} also holds for pseudoconvex domains with the continuous Azukawa metrics since the coincidence $\mu_{\Omega;B}=\mathcal{H}_{\Omega}$ holds for continuous Finsler manifolds $\Omega$.
\end{remark}
\bibliographystyle{plain}
\bibliography{References-1}

\end{document}